\documentclass[12pt]{amsart}

\usepackage{amssymb}
\usepackage{amsmath}
\usepackage{amsfonts}
\usepackage[utf8]{inputenc}
\usepackage{amsthm}
\usepackage{tikz-cd}

\newtheorem{thm}{Theorem}[section]
\newtheorem{lemma}[thm]{Lemma}
\newtheorem{df}[thm]{Definition}
\newtheorem{prop}[thm]{Proposition}
\newtheorem{ex}[thm]{Example}
\newtheorem{cor}[thm]{Corollary}

\theoremstyle{remark}
\newtheorem{remark}[thm]{Remark}

\newcommand{\RR}{\mathcal{R}}
\newcommand{\Z}{\mathcal{Z}}

\newcommand{\R}{\mathbb{R}}
\newcommand{\N}{\mathbb{N}}

\newcommand{\G}{\mathcal{G}}

\title{Sums of even powers of k-regulous functions}
\author{Juliusz Banecki and Tomasz  Kowalczyk}
\date{}

\begin{document}

\maketitle

\begin{abstract}

We provide an example of a nonnegative $k$-regulous function on $\R^n$ for $k\geq 1$ and $n \geq 2$ which cannot be written as a sum of squares of $k$-regulous functions. We then obtain lower bounds for Pythagoras numbers  $p_{2d}(\RR^k(\R^n))$ of $k$-regulous functions on $\R^n$ for $k\geq 1$ and $n\geq 2$. We also prove that the second Pythagoras number of the ring of $0$-regulous functions $\RR^0(X)$ on an irreducible $0$-regulous affine variety $X$ is finite and bounded from above by $2^{\dim X}$.

\end{abstract}
\section*{Introduction}

Let $f$ be a real-valued nonnegative function on a topological space $X$. It is then natural to ask two questions. Can $f$ be written as a sum of squares of functions on $X$ with the same regularity as $f$? And if so, how many summands are needed? In this paper we deal with two instances of this problem. In the first one, $X$ is an irreducible $0$-regulous affine variety and $f$ is a $0$-regulous function. The second case deals with $k$-regulous functions on $\R^n$ for $k>0$.

It was known since Hilbert that not every nonnegative polynomial can be written as a sum of squares of polynomials, however the first explicit example was discovered almost 80 years later by Motzkin (see \cite{Motzkin}). By a famous result of Artin \cite{Artin}, every nonnegative polynomial can be written as a sum of squares of rational functions. Assume that for a nonnegative polynomial $p$ on $\R^n$ in every expression
\begin{equation}\label{eq1}
 p = \sum \left(\frac{f_i}{g_i}\right)^2
\end{equation}
where $f_i, g_i$ are polynomials, there exists a point for which all denominators vanish. In this case, such a point is called a ``bad point". It is known that polynomials with ``bad points" exists (cf. \cite{Delzell}), and therefore it is natural to ask what regularity conditions can be imposed on the rational functions in (\ref{eq1}). We answer this question in the case where $p$ is a positive definite homogeneous polynomial.

 


Let $A$ be a commutative ring with identity.
\begin{df}
We define the $2d$-th Pythagoras number of $A$,  $p_{2d}(A)$ to be the least positive integer $g$ such that any sum of $2d$-th powers of elements of $A$ can be expressed as a sum of $g$ $2d$-th powers of elements of $A$. If such number does not exist, we put $p_{2d}(A)=\infty$.
\end{df}


These invariants were studied before, both in the case of a field of rational functions in one variable \cite{CLPR, Schmid}, and in the case of real holomorphy rings \cite{Becker}. In this paper we will study the Pythagoras numbers and existence of sums of even powers representations in the ring of $k$-regulous functions, the definition of which we  recall below in Section 1. For a more thorough exposition see \cite{4xfr, KN}.

 In general, computing the second Pythagoras number of a ring $A$ is a very difficult task. It is hard to even determine, whether this number is finite or not.
 Take, for example, the field $\R(x_1,x_2,\dots, x_n)$, where only some bounds for the value of the second Pythagoras number $p_2(\R(x_1,\dots, x_n))$ are known. Pfister \cite{Pfister}  was first to show that this number is finite and bounded from above by $2^n$, while the lower bound, which is $n+2$, is attributed to Cassels, Ellison and Pfister (see \cite{grimm} or \cite[Lemma 8.2]{CLDR}). In particular $p_2(\R(x_1,x_2))=4.$ The situation is different when we consider polynomials. It was proven in \cite{CLDR} that  $p_2(\R[x_1, \dots, x_n])=\infty$, for $n\geq 2$. There are also results concerning sums of squares in excellent rings (see \cite{FRS}).

The structure of the paper is as follows. Section 1 contains preliminary results concerning $k$-regulous functions. In Section 2 we discuss the second Pythagoras number $p_2(\RR^0(X))$ of the ring of $0$-regulous functions on an irreducible $0$-regulous affine variety, and provide some upper bounds. In Section 3 we consider $k$-regulous functions on $\R^n$ for $k>0$. An example of a nonnegative $k$-regulous function which cannot be written as a sum of squares is given. Also, some lower bounds for the Pythagoras numbers $p_{2d}(\RR^k(\R^n))$ are established. The last section deals with the problem of representing a quotient of two positive definite forms as a sum of even powers of $k$-regulous functions, for some ~$k$.

\section{Preliminaries}
We start with some preliminary definitions.

\begin{df}
Let $n$ be a positive integer and $k$ be a nonnegative integer. We say that a continuous function $f: \R^n \rightarrow \R$ is $k$-regulous on $\R^n$ if $f$ is of class $\mathcal{C}^k$ and $f$ is a rational function, i.e. there exists a non-empty Zariski open subset $U\subset \R^n$ such that $f|_U$ is regular.
\end{df}

The class of $k$-regulous functions determine a topology on $\R^n$, such that the zero sets of $k$-regulous functions are precisely the closed sets. We call such a topology a $k$-regulous topology. Surprisingly, this topology does not depend on $k$ \cite[Corollaire 6.5]{4xfr}. In particular, $k$-regulous closed subsets of $\R^n$ coincide with algebraically constructible Euclidean closed subsets of $\R^n$.

\begin{df}\label{defin k-reg}
For a $k$-regulous closed subset $X\subset \R^n$ we define the ring of $k$-regulous functions on $X$ as the
quotient
$$\RR^k(X):=\RR^k(\R^n)/\mathcal{I}(X)$$
where $\mathcal{I}(X)$ is the ideal of $k$-regulous functions vanishing on $X$.
\end{df}
Throughout the paper, the term affine $k$-regulous variety  will mean $k$-regulous closed irreducible subset of $\R^n$ for some $n$ (see \cite{4xfr} for more details on the $k$-regulous topology).

We start with some simple results that will be used frequently below.
By $\Z(f)$ we mean the zero set of $f$.
\begin{prop}
Let $X \subset \R^n$ be a $k$-regulous affine variety and $f \in \RR^k(X)$ be a $k$-regulous function on $X$. Then the set $U:= X\setminus\Z(f) $ is isomorphic with a $k$-regulous affine variety.

\begin{proof}
By Definition \ref{defin k-reg} we can find a $k$-regulous function $F \in \RR^k(\R^n)$ such that $F|_X=f$. By noetherianity of the $k$-regulous topology (cf. \cite[Theorem 4.3]{4xfr}), there is a $k$-regulous function $G \in \RR^k(\R^n)$ such that $\Z(G)=X$. Define a $k$-regulous set $$V:=\{(x,y) \in \R^n \times \R \ | \ (F(x)y-1)^2+G(x)^2 =0 \}.$$ Clearly, the sets $U$ and $V$ are isomorphic in the category of $k$-regulous affine varieties.

\end{proof}

\end{prop}

\begin{prop}\label{localization}
Let $X\subset \R^n$ be an affine $k$-regulous variety and $f \in \RR^k(X)$ be a $k$-regulous function on $X$. Then the restriction morphism induces an isomorphism between the localization $\RR^k(X)_f$ and $\RR^k(U)$ where $U:=X\setminus \Z(f)$ is the complement of the zero set of $f$.

\begin{proof}

Let $\phi:\RR^k(X)\rightarrow \RR^k(U)$ be the restriction homomorphism. As $f$ does not vanish on $U$, there is an induced homomorphism
$$\phi_f:\RR^k(X)_f\rightarrow \RR^k(U).$$ We will prove that $\phi_f$ is an isomorphism. Denote by $\pi:\RR^k(\R^n)\rightarrow \RR^k(X)$ the natural projection. If $X=\R^n$, then this is precisely \cite[Proposition 5.6]{4xfr}. Injectivity of $\phi_f$ follows from the irreducibility of $X$.

To show surjectivity, take any $g\in \RR^k(U)$. From the definition of induced topology, there exists a $k$-regulous open subset $V\subset \R^n$ such that $V\cap X=U$. By \cite[Proposition 5.37]{4xfr} we may find a $k$-regulous function $G \in \RR^k(V)$ such that $G|_U=g$. Similarly, there exists a $k$-regulous function $F\in \RR^k(\R^n)$ such that $F|_X=f$. The next step is to modify the functions $G$ and $F$ in order to apply the Łojasiewicz inequality \cite[Lemme 5.1]{4xfr}. Since the $k$-regulous topology is Noetherian, there exists a $k$-regulous function $H \in \RR^k(\R^n)$ such that $\Z(H)=X$. Define modifications of $F$ and $G$ as follows:
$\widetilde{F}=F+H^2$ and $\widetilde{G}=G+H^2|_U$. Clearly, $\widetilde{F}|_X=f$ and $\widetilde{G}|_U=g$. These new functions satisfy the assumptions of the Łojasiewicz inequality, hence, there exists a positive integer $N$, such that $\widetilde{F}^N\widetilde{\G}\in \RR^k(\R^n)$. Let now $\pi(\widetilde{F}^N\widetilde{\G})=h \in \RR^k(X)$, which is a $k$-regulous function on $X$. The function $h$ restricted to the set $U$ is equal to $f^Ng$, therefore $g$ can be written as a fraction $\frac{h}{f^N}$, as $f$ is a unit on $U$. Since $g$ was arbitrary, this finishes the proof.

\end{proof}
\end{prop}

Consequently, the ring $\RR^k(U)$ is a ring of $k$-regulous functions on an affine $k$-regulous variety. Below, we give some examples of the $k$-regulous sets.
\begin{ex}
Consider the real algebraic set $X=\Z(y^2-x^{3+k}+x^2) \subset \R^2$. $X$ is then reducible as a $k$-regulous set. Indeed, there is a decomposition $X=Y\cup W$, where $Y$ is the origin, and $W$ is the union of the 1-dimensional connected components with respect to the Euclidean topology. $W$ is then a zero set of a $k$-regulous function 
$f= 1 - \frac{x^{3+k}}{x^2+y^2}$. Additionally, depending on the parity of $k$, $W$ consists of either one or two curves
\end{ex}

\begin{ex}
Take the real algebraic set $X=\Z(zx^2-y^2)\subset \R^3$, known as the Whitney umbrella. $X$ is an irreducible real algebraic set as well as an irreducible $k$-regulous set (cf. \cite[Exemples 6.12]{4xfr}). Note that $X$ has dimension 2, and has a 1-dimensional "handle" sticking out of its 2-dimensional part.

\end{ex}

\section{Sums of squares of $0$-regulous functions}

In this section we are mainly interested in $0$-regulous functions on $0$-regulous affine varieties. Some results, however, do hold in the case $k>0$. Let $X$ be an affine $k$-regulous variety of dimension $n$.

We start by applying the results of \cite{Mahe} to show the following:

\begin{thm}\label{Mahe}
Any strictly positive $k$-regulous function on $X$ can be written as a sum of at most $2^n$ squares of $k$-regulous functions.
\end{thm}

First we need the following two definitions:

\begin{df}We say that a ring $A$ is a ring of regular functions, if for any element $s$ which is a sum of squares in $A$, the element $1+s$ is invertible (cf. \cite[Chapter 7]{Mahe}).
\end{df}
\begin{remark}
By the weak Nullstellensatz \cite[Proposition 5.9]{3xfr}, the ring $\RR^k(\R^n)$ is a ring of regular functions, for any integers $k \geq 0$ and $n\geq 1$. Clearly, any homomorphic image of a ring of regular functions is a ring of regular functions. Thus the ring $\RR^k(X)$ is a ring of regular functions for any $k$-regulous affine variety.

\end{remark}

Let $A$ be a commutative ring with identity. We denote by $\mathrm{Spec}_r(A)$ the real spectrum of $A$ (cf. \cite[Chapter 7]{Bochnak}).

\begin{df}
We say that an element $f \in A$ is totally positive if for every element $\alpha \in \mathrm{Spec}_r(A)$, we have $f(\alpha)>0$.
\end{df}

Now, in order to obtain Theorem \ref{Mahe} as a reformulation of \cite[Theorem 7.3]{Mahe} it suffices to show the following:

\begin{prop}\label{totally positive}
Any strictly positive $k$-regulous function $f$ on $X$ is totally positive in the ring $\RR^k(X)$.
\end{prop}

\begin{proof}
Let $G \in \RR^k(\R^m)$ be a $k$-regulous function such that $\Z(G)=X$. The natural projection $\RR^k(\R^m) \rightarrow \RR^k(X)$ allows us to identify $\mathrm{Spec}_r(X)$ with the set $\{ \alpha \in \mathrm{Spec}_r(\R^m)\ | \ G(\alpha)=0 \}$. Let $F$ be any $k$-regulous function on $\R^m$ so that $F|_X=f$. In order to prove that $f$ is totally positive it is enough to show that the set $\widetilde{S_1}=\{\alpha \in \mathrm{Spec}_r(X) \ | \ f(\alpha)\leq 0\}$ is empty. This in fact is equivalent to emptiness of the set $\widetilde{S_2} = \{ \alpha \in \mathrm{Spec}_r(\R^m) \ | \ G(\alpha)=0, \ F(\alpha)\leq 0 \}$. The result now follows readily from the $k$-regulous version of the Artin-Lang Property \cite[Proposition 5.6]{3xfr}.

\end{proof}



We may now state the main result of this section.
\begin{thm}\label{0=regulous squares}
Let $X$ be an affine $0$-regulous variety of dimension $n$. 
The second Pythagoras number $p_2(\RR^0(X))$ of $0$-regulous functions on $X$ satisfies
$$p_2(\RR^0(X))\leq 2^n.$$

\end{thm}

\begin{remark}
Theorem 6.1 in \cite{3xfr} says that any nonnegative $0$-regulous function on $\R^n$ is a sum of squares of $0$-regulous functions. The proof below gives more, namely: any nonnegative $0$-regulous function on a $0$-regulous affine variety is a sum of squares of $0$-regulous functions.
\end{remark}
\noindent
\emph{Proof of Theorem \ref{0=regulous squares}}
Let $f \in \RR^0(X)$ be a sum of squares. Denote by $Z=\Z(f)$ the zero set of $f$ and let $U=X\setminus Z$ be its complement. Consider now the restriction of $f$ to the set $U$. Clearly, $f$ is a totally positive element in $\RR^0(U)$. Hence by Theorem \ref{Mahe} we get $f|_U = \sum_{i=1}^{2^n}f_i^2$, with $f_i \in \RR^0(U)$ for any $i=1,2,\dots, 2^n$. By Proposition \ref{localization}, we have an isomorphism $\RR^0(U)\cong \RR^0(X)_f$, and thus we can write $f_i=\frac{g_i}{f^{l_i}}$ for some $g_i \in \RR^0(X)$ and $i=1,2,\dots, 2^n$. Without loss of generality we may assume that $l_1=l_2=\dots=l_{2^n}=N$ for some positive integer $N$.

For each $i=1,2,\dots, 2^n$ define the function
\begin{equation*}
    h_i:=
    \begin{cases}
        \frac{g_i}{f^N} \text{ on } U \\
        0 \text{ on } Z
    \end{cases}
\end{equation*}
A straightforward calculation shows that each function $h_i$ is continuous on $X$ and 
$$f=\sum_{i=1}^{2^n}h_i^2$$
on $X$.
To finish the proof it is enough to show that the functions $h_i$ for $i=1,2,\dots, 2^n$ are $0$-regulous on $X$, but this follows from the lemma below.
\qed

\begin{lemma}
Let $X\subset\R^n$ be an affine 0-regulous variety and $f,g\in\RR^0(X)$ be $0$-regulous functions such that
\begin{equation*}
    Z=\{x\in X:f(x)=0\}\subset\{x\in X:g(x)=0\}.
\end{equation*}
If the function $h:X \rightarrow \R$ defined as
\begin{equation*}
    h:=
    \begin{cases}
        \frac{g}{f} \text{ on } X \backslash Z \\
        0 \text{ otherwise}
    \end{cases}
\end{equation*}
is continuous, then $h \in\RR^0(X)$.
\begin{proof}
First, note that $g$, $f$ and $X$ are semialgebraic and so is $h$. 
Clearly we have $\Z(f)\subset \Z(h)$. Applying the semialgebraic version of the Łojasiewicz inequality \cite[Theorem 2.6.6]{Bochnak} and \cite[Proposition 2.6.8]{Bochnak} we can find positive rational constants $\epsilon,c,p$ such that 
\begin{equation*}
    |h| < c|f|^\epsilon(1+||x||^2)^p 
\end{equation*}
on $X \backslash Z$.
Let now $F,G,E\in\RR^0(\R^n)$ be any $0$-regulous functions satisfying $f=F|_X$,  $g=G|_X$ and $\mathcal{Z}(E)=X$.
Consider the set
\begin{equation*}
    V:=\{x\in\R^n:|G(x)| \geq c|F(x)|^{1+\epsilon}(1+||x||^2)^p \}.
\end{equation*}
Note that $V$ is a closed semialgebraic set and $V\cap X=Z$. Therefore, by applying the Łojasiewicz inequality again, we can find a continuous semialgebraic function $K:V\rightarrow \R$ and a positive integer $m>0$ such that:
\begin{equation*}
   K E^2=F^{2m-1}
\end{equation*}
on $V$.
Finally, consider the function
\begin{equation*}
    H:=\frac{GF^{2m-1}}{F^{2m}+E^2}.
\end{equation*}
We claim that $H\in \RR^0(\R^n)$. It is sufficient to show that it can be continuously extended to its indeterminacy locus, i.e. $Z$. On the set $V$ we have the inequality $|H|\leq |GK|$ which holds near $Z$. On the complement of $V$ we have $|H|\leq |G/F|< c|F|^\epsilon(1+||x||^2)^p$. In both cases $H$ approaches zero at every point of $Z$, so indeed it is continuous after extending by 0 on $Z$, and thus $0$-regulous.
Finally, $HF=G$ on $X$, so $H|_X=h$. This finishes the proof.
\end{proof}
\end{lemma}

\begin{remark}
Reasoning from the proof of Theorem  \ref{0=regulous squares} cannot be applied to the $k$-regulous functions on $\R^n$ for $k\geq 1$ and $n\geq 2$, as there is no guarantee that the extensions of the functions $f_i$ will be of class $\mathcal{C}^k$ .
\end{remark}
\begin{cor}
Since the field $\R(x_1,\dots, x_n)$ is a field of fractions of $\RR^0(\R^n)$ we obtain

$$n+2 \leq p_2(\RR^0(\R^n))\leq 2^n.$$
In particular $p_2(\RR^0(\R^2))=4$.
\end{cor}
For $n=2$, Theorem \ref{0=regulous squares} and \cite[Theorem 6.1]{3xfr} immediately imply

\begin{cor}
Let $f \in \RR^0(\R^2)$. Then $f$ is nonnegative on $\R^2$ if and only if there exist functions $f_1, f_2, f_3, f_4 \in \RR^0(\R^2)$ such that $f=f_1^2+f_2^2+f_3^2+f_4^2$.
\end{cor}

\section{Sums of even powers of $k$-regulous functions}

In this section, we study nonnegative elements and sums of even powers in the rings of $k$-regulous functions on $\R^n$ for $k\geq 1$ and $n\geq 2$. We provide an example of a nonnegative $k$-regulous function which is not a sum of squares, in contrast with the case $k=0$. Also, lower bounds for $2d$-th Pythagoras numbers are given. 
\subsection{Nonnegative functions and the set of bad points.}

Let $A$ be commutative ring with identity. We say that an element $a \in A$ is positive semidefinite (psd) if for every $\alpha \in \mathrm{Spec}_r A$ we have $a(\alpha)\geq 0$. By the $k$-regulous version of Artin-Lang property \cite[Proposition 5.6]{3xfr}, any nonnegative $k$-regulous function on $\R^n$ is psd in the ring $\RR^k(\R^n)$. In real algebraic geometry one can consider the following problem: is every psd element a sum of squares (sos) (see \cite{FRS} for results on excellent rings) in a given ring? Below, we prove the existence of so-called ``bad points" for $k$-regulous functions on $\R^n$. This proves that for $\RR^k(\R^n)$  we have psd $\neq$ sos, provided $k\geq1$ and $n\geq 2$.

\indent
Recall the following:

\begin{df}
Let $f$ be a real-valued continuous function on $\R^n$ and $k$ be a positive integer. We say that $f$ is homogeneous of degree $k$ if $f(\alpha x)=\alpha^k f(x)$ for all nonzero real numbers $\alpha$ and any $x \in \R^n$.
\end{df}



We state the main theorem of this section.

\begin{thm}\label{Sos l}
Let $f$ be a $k$-regulous function on $\R^n$ which is homogeneous of degree $2kd$. If $f$ is a sum of $l$ $2d$-th powers of $k$-regulous functions, then $f$ is a sum of $ l$ $2d$-th powers of polynomials. In particular, $f$ is a polynomial.
\end{thm}

\begin{proof}
We have a presentation $f=\sum_{i=1}^l f_i^{2d}$, for some $k$-regulous functions $f_i \in \RR^k(\R^n)$. Clearly, order of $f_i$ at the origin is necessarily at least $k$, for $i=1,2,\dots, l$. Thus we can write $f_i=P_i+h_i$ where $P_i$ is the Taylor polynomial of degree $k$ of $f_i$ and $h_i$ is a $k$-regulous function of order at least $k+1$ at the origin (cf. \cite[Lemma 2]{czarnecki}). Consequently,
$$f=\sum_{i=1}^l f_i^{2d}=\sum_{i=1}^l (P_i + h_i)^{2d}.$$
After rearranging terms we obtain
$$f-\sum_{i=1}^l P_i^{2d} = G$$
where $G$ is a $k$-regulous function of order at least $2kd+1$. Consequently, we get a homogeneous function of degree $2kd$, and of order at least $2kd+1$, hence a zero function. In particular, a presentation of $f$ as a sum of $2d$-th powers of $k$-regulous functions implies a presentation of $f$ as a sum of $2d$-th powers of polynomials. Thus, $f$ is a polynomial.

\end{proof}







It was asked in \cite{3xfr} whether any nonnegative $k$-regulous function on $\R^n$ is a sum of squares of $k$-regulous functions. We can now show that this is not the case.

\begin{thm}

Let $n = 2$ and $k\geq 1$. Take the $k$-regulous function $$f_k=\frac{x^{2+2k}}{x^2+y^2}\in \mathcal{R}^k(\R^2).$$ Consider the set
$$\mathrm{Denom}(f_k):= \{ h\in \mathcal{R}^k(\R^2):h^2f_k=\sum_{i=1}^l g_i^2 , \,    g_i \in \mathcal{R}^k(\R^2)  \}.$$

Let $I_k$ be the ideal generated by $\mathrm{Denom}(f_k)$.
Then the set of ``bad points", $\mathrm{Bad}(f_k):=\Z(I_k)$ is nonempty.
\begin{proof}

Obviously, $x^2 +y^2 \in \mathrm{Denom}(f_k)$, hence $\Z(I_k)\subset \Z(x^2+y^2)$.
Suppose that the set $\Z(I_k)$ is empty. Then, there exists $h\in\mathrm{Denom}(f_k)$ with a nonzero value at the origin. A simple calculation shows that $x^2+y^2+h^2\in\mathrm{Denom}(f_k)$. Since $\mathrm{Denom}(f_k)$ is closed under multiplication by elements of $\mathcal{R}^k(\R^n)$ and the given element is a unit, we get that $1 \in \mathrm{Denom}(f_k)$.
This is an immediate contradiction with Theorem \ref{Sos l} as the function $f_k$ is homogeneous of degree $2k$ and it is not a polynomial.

\end{proof}

\end{thm}

The function $f_k$ might as well be considered a function of $n$ variables for any $n \geq 2$. Using the above argument we see that $\mathrm{Bad}(f_k)=\Z(x^2+y^2)$. As a consequence
\begin{thm}
The function $f_k$ is not a sum of squares in $\mathcal{R}^k(\R^n)$ for any $n \geq 2$.
\end{thm}

Note that the function $f_k$ is actually $(2k-1)$-regulous.
\begin{remark}
In Dellzell's PhD Thesis \cite{Delzell} it is shown that the set of ``bad points" of a nonnegative polynomial on $\R^n$ has codimension at least 3. The above result shows that the set of ``bad points" of a nonnegative $k$-regulous function on $\R^n$ for $k\geq 1 $ and $n\geq 2$ can have codimension 2. 
\end{remark}


 
 \subsection{Pythagoras numbers of the rings $\RR^k(\R^n)$}
 Consider the vector space $ \R[x_1,x_2,\dots,x_n]_{2kd}$ of homogeneous polynomials of degree $2kd$ in $n$ variables. Elements of this vector space are also called $n$-ary $2kd$-ics. Even though the Pythagoras numbers are usually defined for rings, it makes perfect sense to define them for the above vector space. 
 
 \begin{df}

We define the $2d$-th Pythagoras number of the vector space $ \R[x_1,x_2,\dots,x_n]_{2kd}$, $p_{2d}(n,2kd)$ to be the least positive integer $g$ such that any sum of $2d$-th powers of $n$-ary $k$-ics can be presented as a sum of at most $g$ $2d$-th powers of $n$-ary $k$-ics. If such number does not exist, we put $p_{2d}(n,2kd)=\infty$.
 \end{df}
 
It follows from the Caratheodory Theorem \cite[Proposition 2.3]{Reznick 2} that the following formula holds
$$   p_{2d}(n,2kd)\leq {{n+2kd-1}\choose{n-1}}$$
(see \cite[Section 4]{Reznick 2} for the $p_{2d}(n,2d)$ which is $T(n,2d)$ in Reznick's notation).

\begin{thm}\label{k-reg p2d}
Let $k,d\geq 1$ and $n\geq 2$ be positive integers. Then the $2d$-th Pythagoras number of $\RR^k(\R^n)$ satisfies
$$p_{2d}(n,2kd)\leq p_{2d}(\RR^k(\R^n)).$$
\begin{proof}
Let $f\in \R[x_1,x_2,\dots,x_n]$ be a homogeneous polynomial of degree $2kd$ which can be written as a sum of $p_{2d}(n,2kd)$ $2d$-th powers of real polynomials, but no fewer. Application of Theorem \ref{Sos l} to $f$ finishes the proof.
\end{proof}
\end{thm}
Further properties of the numbers $p_{2d}(n,2kd)$ are yet to be established.
\begin{remark}\label{remark d=1}
Let $d=1$ and $n>2$. It is known that the set $\{p_2(n, 2k) \ | \ k \in \N \}$ is infinite. This is because $p_2(\R[x_1,x_2,\dots, x_n])=\infty$ for $n\geq 2$. In particular for a fixed $n\geq 3$ we have
$$\limsup_{k\rightarrow \infty} p_2(\RR^k(\R^n))=\infty.$$
 Also, since $\R(x_1, x_2, \dots, x_n)$ is the field of fractions of $\RR^k(\R^n)$, we have the following inequality
$$\max (n+2, p_2(n,2k))\leq  p_2(\RR^k(\R^n)).$$

Another important case occurs for $k=1$. By  \cite[Section 4]{Reznick 2} we have a lower bound given by
$${n+d-1\choose{n-1}} \leq p_{2d}(\RR^1(\R^n)).$$

\end{remark}

Also, the results of Grimm \cite{grimm} imply the following
$$ \max (n+1, p_{2d}(n,2kd))\leq  p_{2d}(\RR^k(\R^n)).$$

\begin{remark}

It is worth noting that for $n=2$ and $d=1$ the Theorem \ref{k-reg p2d} gives us nothing since $p_2(2,2k)=2$ for any $k>0$.

However, as a consequence of \cite{Cassels}, the Motzkin polynomial $$M(x,y)=1+x^2y^2(x^2+y^2-3)$$ can be written as a sum of 4 squares of regular functions, but not less, hence $$4 \leq p_2( \RR^k(\R^2))$$
for $k\geq 0$.
\end{remark}
 
 \begin{remark}
Aside from asking whether the numbers $p_{2d}(\RR^k(\R^n))$ are finite for $n\geq 3$, $k\geq 1$ and $d>1$ it is also interesting to ask: does 
 $$ \limsup_{k\rightarrow \infty}p_{2d}(\RR^k(\R^n)))=\infty$$
hold?

 This would easily follow from the equality
 $$ \limsup_{k\rightarrow \infty}p_{2d}(n, 2kd)=\infty.$$
 The last limit is equivalent to $p_{2d}(\R[x_1,\dots, x_{n-1}])=\infty$. The problem of computation of $p_{2d}(\R[x_1,\dots, x_n])$ is listed as a Problem 8 in \cite{CLDR}. To the authors' best knowledge, this problem still remains open.
For $d=1$, the above holds.
 
 \end{remark}

\section{Case of positive definite forms}
In the last section, we deal with the problem of representing a quotient of two positive definite forms as a sum of $2d$-th powers of $k$-regulous functions.
Let us recall that a homogeneous form of degree $2d$ is positive definite if it vanishes only at the origin.
We start with

\begin{lemma}\label{lemma}
Let $q$ be a positive definite form of degree $2d$. Let $x^\alpha$ be a monomial in $x_1,x_2, \dots, x_n$ of degree k, such that $k-2d-1 \geq 0$. Then the function $\frac{x^\alpha}{q}$ is $(k-2d-1)$-regulous.
\begin{proof}
This follows from an easy application of spherical coordinates.
\end{proof}
\end{lemma}

\begin{thm}\cite[Third Theorem]{Reznick}\label{Reznick}
Let $p$ and $q$ be positive definite forms on $\R^n$. Assume that the rational function $\frac{p}{q}$ is homogeneous of degree $2kd$. Then, it can be written as a sum of $2d$-th powers of rational functions whose common denominator is a product of powers of $q$ and $(x_1^2+x_2^2+\dots + x_n^2)$.
\end{thm}

These two results immediately imply the following.
\begin{thm}
With the assumptions of Theorem \ref{Reznick}, the rational function $\frac{p}{q}$ can be written as a sum of $2d$-th powers of $(k-1)$-regulous functions on $\R^n$. In particular, $\frac{p}{q}$ is a sum of squares of functions of class at least $\mathcal{C}^{kd-1}$.
\end{thm}
\begin{proof}

By Theorem \ref{Reznick}, we obtain a presentation $\frac{p}{q}=\sum_{i=1}^N f_i^{2d}$, for some positive integer $N$. For every $i=1,2,\dots, N$, $f_i$ is a homogeneous rational function of degree $k$ with a positive definite denominator. Hence, by Lemma \ref{lemma}, each $f_i$ is at least of class $\mathcal{C}^{k-1}$.

\end{proof}

\begin{remark}
In \cite{Kucharz 2} Kucharz introduced the notion of a $k$-regulous Nash function. A straightforward computation shows that the results of this paper also hold for $k$-regulous Nash functions on $\R^n$.
\end{remark}

In their paper \cite{Laraki}, Laraki and Lasserre considered a certain type of an optimization problem for rational functions on a basic semialgebraic set $K$, and its application to various finite games, such as standard static games, Loomis games and finite absorbing games. While they require the denominators of these rational functions to be strictly positive on $K$, we believe that studying sums of squares of $k$-regulous functions can lead to further developments in optimization and game theory.

\section{Acknowledgment}
The authors are grateful to our dear friend Maciej Zieliński for fruitful discussions. We thank the anonymous referees for their useful comments and remarks, which greatly improved the quality of this paper.


\begin{small}
\noindent
Juliusz Banecki

\noindent
Institute of Mathematics

\noindent
Faculty of Mathematics and Computer Science

\noindent
Jagiellonian University

\noindent
ul. Łojasiewicza 6, 30-348 Kraków, Poland

\noindent
e-mail: juliusz.banecki@autonomik.pl 

\end{small}

\vspace{6pt}

\begin{small}
\noindent
Tomasz Kowalczyk

\noindent
Institute of Mathematics

\noindent
Faculty of Mathematics and Computer Science

\noindent
Jagiellonian University

\noindent
ul. Łojasiewicza 6, 30-348 Kraków, Poland

\noindent
e-mail: tomek.kowalczyk@uj.edu.pl

\end{small}

\end{document}